\documentclass[12pt,a4paper]{amsart}

\usepackage{amsmath}
\usepackage{amssymb}

\theoremstyle{plain}   %% This is the default, anyway
\begingroup % Confine the \theorembodyfont command
\newtheorem*{Proposition S}{Proposition S}

\newtheorem{theorem}{Theorem}[section]   % Numbered within each section
\newtheorem{lemma}[theorem]{Lemma}         % Numbered along with thm
\newtheorem{proposition}[theorem]{Proposition}  % Numbered along with thm
\endgroup

\theoremstyle{definition}
\newtheorem{definition}[theorem]{Definition}   % Numbered along with thm 

\theoremstyle{remark}
\newtheorem{remark}[theorem]{Remark}        % Numbered along with thm
\numberwithin{equation}{section}

\newcommand{\ep}{\varepsilon}

\newcommand{\R}{{\mathbb R}}
\newcommand{\N}{{\mathbb N}}

\newcommand{\D}{\cal D}

\newcommand{\cal}{\mathcal}

\newcommand{\lin}{\operatorname{span}}

\newcommand{\vf}{\varphi}

\newcommand{\interior}{\operatorname{int}}
\newcommand{\supp}{\operatorname{supp}}

\newcommand{\graph}{\operatorname{graph}}

\newcommand{\diam}{\operatorname{diam}}

\newcommand{\osc}{\operatorname{osc}}

\usepackage[dvips]{color}

\begin{document}

\title{On sets of discontinuities  of functions continuous on all lines}

\thanks{}

\author{Lud\v ek Zaj\'\i\v{c}ek}



\email{zajicek@karlin.mff.cuni.cz}

\address{Charles University,
Faculty of Mathematics and Physics,
Sokolovsk\'a 83,
186 75 Praha 8-Karl\'\i n,
Czech Republic}

%\date{\today}

\begin{abstract} 
Answering a question asked by K.C. Ciesielski and T. Glatzer in 2013, we construct a $C^1$-smooth
 function $f$ on $[0,1]$ and a set $M \subset \graph f$ nowhere dense in $\graph f$ such that there 
 does not exist any linearly continuous
  function  on $\R^2$ (i.e. function continuous on all lines) which is discontinuous at each point of
	 $M$. We substantially use a recent full characterization of sets of discontinuity points of
	 linearly continuous functions on $\R^n$ proved by T. Banakh and O. Maslyuchenko in 2020.
	As an easy consequence of our result, we prove that the necessary condition for such sets of discontinuities proved by
	 S.G. Slobodnik in 1976 is not sufficient. We also prove an analogon of this Slobodnik's result 
	 in separable Banach spaces.
	\end{abstract}
   
\markboth{L.~Zaj\'{\i}\v{c}ek}{Functions continuous on all lines}

\maketitle

\section{Introduction}

Separately continuous functions on $\R^n$ (i.e. functions continuous on all lines parallel
 to an coordinate axis) and also linearly continuous functions (i.e. functions continuous on all lines) were investigated in a number of articles, see the survey \cite{CM}. Note that linearly
 continuous functions are well-defined in any linear space  and recent articles \cite{Za} and \cite{BM}
 investigate them also in Banach (and even more general) spaces.

It appears that linearly continuous functions are much ``more close to continuous functions''
 than separately continuous functions.
First note that, by Lebesgue's result of \cite{Le},
\begin{equation}\label{leb}
\text{each separately continuous function on $\R^n$ belongs to the $(n-1)$-th Baire class}
\end{equation}
 (and that the number $n-1$  is optimal, cf. \cite{CM}). On the other hand, it was proved independently 
 (answering a question posed in  \cite{CM})  in  \cite{Za} and \cite{BM}
 that each linearly continuous function on $\R^n$ belongs to the first Baire class.

A natural question how small must be members of the families (where $D(f)$ denotes the set of discontinuity points of $f$)
$$ \D_s^n:= \{ D(f):\ f\  \text{is  a separately continuous function on}\ \R^n\}$$ 
and 
$$ \D_l^n:= \{ D(f):\ f\  \text{is  a linearly continuous function on}\ \R^n\}$$ 

was considered in several works, see \cite{CM}.
 Clearly,  $D_l^n \subset \D_s^n$ and each set from $\D_s^n$ is an $F_{\sigma}$ set. It appears that the 
 sets from $D_l^n$ ``must be essentially smaller'' than those from $\D_s^n$.

The following complete characterization of sets  from $\D_s^n$ was given
	 in \cite{Ker} (and was proved independently by another method in \cite{Sl}).
	
	\begin{theorem}\label{Ker} (R. Kershner, 1943)
	
	A set $M \subset \R^n$ belongs to $\D_s^n$
	 if and only if $M$ is an $F_{\sigma}$ set and the orthogonal projection of $M$ onto each $(n-1)$-dimensional
	 coordinate hyperplane is a first category (= meager) set.
	\end{theorem}

	This characterization
	 shows that each member of $\D_s^n$ is a first category set, but it can have positive Lebesgue measure (even
	 its complement can be Lebesgue null, cf. \cite{CM}). On the other hand (see Remark \ref{oslob} (a) below) all members of 
	  $D_l^n$ are Lebesgue null.
		
		Probably the first result concerning the system  $D_l^n$ was published in 1910 by W.H Young and G.C. Young
		\cite{YY}; they constructed a linearly continuous function
 on $[0,1]^2$ for which $D(f)$ is uncountable in every nonempty open set.

A.S. Kronrod in 1945 (see \cite[p. 268]{Sh}, \cite[p. 28]{CM}) considered the natural problem to find a complete characterization
 of sets from the system  $D_l^2$. 

As a partial solution of ($n$-dimensional) Kronrod's problem, Slobodnik   proved a theorem (\cite[Theorem 6]{Sl})
  whose obvious reformulation reads as follows.

\begin{theorem}\label{Slob} (Slobodnik, 1976)
Let $M \in D_l^n$. Then we can write  $M = \bigcup_{k=1}^{\infty} B_k$, where each $B_k$ has the following properies.
\begin{enumerate}
\item  $B_k$ is a compact subset of a  Lipschitz hypersurface $L_k$.

\item The orthogonal projection of $B_k$ onto each $(n-1)$-dimensional
	  hyperplane $H \subset \R^n$ is nowhere dense in $H$.
\item  For each $c \in \R^n \setminus B_k$, the set $\{\frac{x-c}{\|x-c\|}:\ x \in B_k\}$ is nowhere dense in
 the unite sphere $S_{\R^n}$.
\end{enumerate}
\end{theorem}

\begin{remark}\label{oslob}
\begin{enumerate}
\item [(a)]   
For the definition of a Lipschitz hypersurface see Definition \ref{liphyp}. Property (i) clearly implies that $M \subset \R^n$ is Lebesgue null.
\item [(b)]
The article \cite{Sh} (written independently on \cite{Sl}) contains results which are very close to Theorem \ref{Slob} with $n=2$.
\item [(c)]
Conditions (i) and (ii) clearly imply that in (i) we can write that $B_k$ is nowhere dense in $L_k$.
\item[(d)]
An equivalent reformulation of (iii) (used in \cite{Sl}) is the following: 

(iii)*\ \  For each $c \in \R^n \setminus B_k$ and each hyperplane $H \subset \R^n \setminus \{c\}$, the central
projection  from $c$ of $B_k$ onto  $H$ is nowhere dense in $H$.
\end{enumerate}
\end{remark}

Further interesting contributions to Kronrod's problem were proved in  \cite{CG1}  and 
 \cite{CG2}.
 Main results of \cite{CG1} read as follows. 

\begin{theorem}\label{CG}\  (K.Ch. Ciesielski and T. Glatzer, 2012)

\begin{enumerate}
\item If $n\geq 2$ and $f: \R^{n-1} \to \R$ is convex and $M \subset \graph f$ is nowhere dense in $\graph f$, then
 there exists a linearly continuous function $g$ on $\R^n$ such that $M \subset D(g)$.
\item If  $f: \R \to \R$ is $C^2$  smooth and $M \subset \graph f$ is nowhere dense in $\graph f$, then
 there exists a linearly continuous function $g$ on $\R^2$ such that $M \subset D(g)$.
\item  There exists  $f: \R \to \R$ having bounded derivative 
  and $M \subset \graph f$ which is nowhere dense in $\graph f$
such that there does not
 exist any linearly continuous function $g$ on $\R^2$ such that $M \subset D(g)$.
\end{enumerate}
\end{theorem}

The article \cite{CG2} contains a full characterization of sets from  $D_l^2$. However, this solution
 of Kronrod's problem (in $\R^2$) is not quite satisfactory (cf. \cite[p. 29]{CM}) since it uses the topology on the set of all lines in $\R^2$ (and its applicability is unclear).

A nice applicable solution of Kronrod's problem in $\R^n$  was proved by T. Banakh and O. Maslyuchenko in \cite{BM}. It asserts that
 a subset of $ \R^n$ belongs to $D_l^n$  if and only if it is ``$\overline{\sigma}$-$l$-miserable'' (see
 Subsection  \ref{bama} for details). 
In \cite{BM}, several applications of this characterization are shown and other two applications are contained in the present article. 

In Section  \ref{mare} we use the Banakh-Maslyuchenko characterization as the main ingredience in the proof of the main result of the present article
 (Theorem \ref{hlavni}) which shows, answering a question from \cite{CG1}, that the function $f$ from
 Theorem \ref{CG} (iii) can be even $C^1$-smooth. More precisely, we use this characterization in the proof
 ot the basic Lemma \ref{DP}. It seems that any proof of Theorem \ref{hlavni} based on Lemma \ref{DP} needs a nontrivial inductive construction. The idea of our construction based on Lemma \ref{tlgr} and Lemma \ref{darb} is not difficult, but
 the detailed proof is unfortunately rather long and slightly technical.

As an easy but interesting consequence of our Theorem \ref{hlavni}, we obtain in Section \ref{SNCINS} that Slobodnik's necessary condition
 for sets from  $D_l^n$ is not sufficient (which  supports the opinion that  there exists no characterization of sets from $\cal D_l^n$ similar to Kershner's characterization of  sets from $\cal D_s^n$). 

In Section  \ref{SNC}, we prove an analogon of Slobodnik's result in separable Banach spaces 
for functions having the  Baire property (which
 improves \cite[Corollary 4.2]{Za}). This 
 result (Proposition \ref{anslo}) which easily implies Slobodnik's theorem in $\R^n$
  is an easy consequence of \cite[Corollary 4.2]{Za} and the Banakh-Maslyuchenko characterization.

\section{Preliminaries} 

\subsection{Basic notation}

In the following, by a Banach space we mean a real Banach space with a norm $\|\cdot\|$. If $X$ is a Banach space, we set
    $S_X:= \{x \in X: \|x\|=1\}$.  
		By $C[0,1]$ and  $C^1[0,1]$ we denote the set of all continuous and $C^1$-smooth functions on $[0,1]$, respectively.
		 If $f \in C[0,1]$, then $\|f\|$ always denotes the supremum norm of $f$.
		
		The symbol $B(x,r)$ will denote the open ball with center $x$ and radius $r$. 
		We say that a set $Q$ is an $\ep$-net of a subset $A$ of a metric space, if 
		$Q \subset A \subset \bigcup_{x\in Q} B(x,\ep)$.

		The oscillation of a function $f$ on a set $M$ will be denoted by $\osc(f,M)$. By $\graph f$, $\supp f$ and $D(f)$, we denote the graph,
		the  closed support $\overline{\{x: f(x) \neq 0\}}$ and the set of discontinuity points of $f$, respectively.
		We will write  $f_n  \rightrightarrows f$ if the sequence $(f_n)$ uniformly converges to $f$.
		 The Lebesgue measure on $\R$ is denoted by $\lambda$.

In a metric space $X$, the system of all sets with the Baire property is the smallest $\sigma$-algebra containing all open sets and all first category sets. We  say that a function $f$ on $X$ has the Baire property 
if $f^{-1}(B)$ has the Baire property for all Borel sets $B \subset Y$ 
(see \cite[$\S$ 32]{Ku}).
We will use the following definition.  
\begin{definition}\label{liphyp}
Let $X$ be a Banach space. We say that $A \subset X$ is a Lipschitz  hypersurface 
 if there exists a $1$-dimensional linear space $V \subset X$, its topological complement
 $Y$ and a Lipschitz mapping $\vf: Y \to V$ such that $A= \{y + \vf(y):\ y \in Y\}$.
\end{definition}
It is easy to see that each Lipschitz  hypersurface is a closed set and that, if $X = \R^n$, then we can
 demand that $Y$ is an orthogonal complement of $V$. 

\subsection{Notation and two lemmas concerning tangent lines}

If $f \in C^1[0,1]$, $z \in [0,1]$ and $Z \subset [0,1]$,  then we use the following notation.
\begin{enumerate}
\item 
By $ A_{f,z}$ we denote the affine function  
$$ A_{f,z} (x)= f(z) + f'(z)(x-z),\ x \in \R.$$
\item
 By $T_{f,z}$ we denote the tangent line to $\graph f$ at the point $(z, f(z))$, i.e. 
 $$T_{f,z}:=  \graph (A_{f,z}).$$
\item
We set 
$$T_{f,Z}:= \bigcup_{z \in Z} T_{f,z}.$$
\end{enumerate}

We will need the following easy lemma.

\begin{lemma}\label{stkon}
Let   functions $f$ and $f_1,f_2,\dots$ belong to $C^1[0,1]$.
Suppose that $z$ and $z_1,z_2\dots$ belong to $[0,1]$, $(x,y) \in \R^2$
  and
$$  f_n  \rightrightarrows f,\ \ \  f'_n \rightrightarrows f',\ \ \ z_n \to z.$$
 Then the conditions $(x,y) \in  T_{f_n,z_n}$,\ $n=1,2,\dots$,  imply  $(x,y) \in  T_{f,z}$.
\end{lemma}
\begin{proof}
 By 
the 							assumptions, we have  $f_n(z_n) + f'_n(z_n) (x-z_n) =y$, $ n\in \N$. Since
 $f_n(z_n) \to f(z)$ and $f'_n(z_n) \to f'(z)$ (see, e.g., \cite[Theorem 7.5, p. 268]{Du}), we obtain 
 $f(z) + f'(z) (x-z) =y$.
\end{proof}

The following lemma is also rather easy but is an important ingredient in our proof of Theorem \ref{hlavni}.

\begin{lemma}\label{tlgr}
Suppose that $f\in C^1[0,1]$ and $f'$ has infinite variation on an interval $[\alpha, \beta] \subset [0,1]$.
 Then
  there exist numbers $e$, $w$ such that  $\alpha < e < w < \beta$ and $(w,f(w)) \in T_{f,e}$.
\end{lemma}
\begin{proof}
First observe that there exist numbers  $e_0$, $w_0$, $e_1$, $w_1$ such that
 $\alpha <e_0 < w_0 < \beta$, $\alpha <e_1 < w_1 < \beta$    and 
\begin{equation}\label{ctyrib}
f(w_0) < f(e_0) + f'(e_0) ( w_0- e_0),\ \ \ f(w_1) > f(e_1) + f'(e_1) ( w_1- e_1).
\end{equation}
To construct $e_0$ and  $w_0$, note that $f'$ is not nondecreasing on $(\alpha, \beta)$ and thus we can
 choose numbers  $\alpha < e^* < w_0 < \beta$ with  $f'(e^*) > f'(w_0)$. Set 
$e_0:= \max \{x\in [e^*,w_0]:\ f'(x)= f'(e^*)\}$. Then we have
$$  f(w_0) = f(e_0) + \int_{e_0}^{w_0} f' < f(e_0) + f'(e_0) (w_0-e_0)$$
 and so $e_0$ and  $w_0$ satisfy \eqref{ctyrib}. The existence of $e_1$ and  $w_1$ follows quite analogously.

Now set 
$$ e(t):= t e_1 + (1-t) e_0,\ \ w(t):= t w_1 + (1-t) w_0,\ \ \ t\in [0,1].$$
 Then clearly  $e(0) =e_0, \ w(0)= w_0,\ e(1)= e_1,\ w(1)= w_1$ and  $e(t)<w(t),\ t\in [0,1]$.
 The function $g(t):= f(w(t)) - f(e(t)) - f'(e(t)) (w(t)- e(t)),\ t\in [0,1],$ is clearly continuous,
 $g(0) <0$ and $g(1)>0$. Consequently there exists $t^* \in (0,1)$ such that $g(t^*)=0$ and so 
 $e:= e(t^*)$ and $w:= w(t^*)$ have the required property.
\end{proof}

\subsection{Banakh-Maslyuchenko characterization}\label{bama}

The authors of \cite{BM} work in ``Baire cosmic vector spaces'' but we work in the present article
 in the more special context of  separable Banach spaces; so we present  basic definitions  from \cite{BM}
 in Banach spaces only.

\begin{definition}\label{miser}
Let $X$ be a Banach space and $A \subset X$.
\begin{enumerate}
\item
A set $V\subset X$  is called an $l$-neighborhood of  $A$ if for any
 $a \in A$ and $v \in X$ there exists $\ep>0$ such that $a +  [0,\ep) \cdot v \subset V$.
\item   $A$ is called $l$-miserable if $A \subset \overline{X \setminus L}$ for some closed
$l$-neighborhood $L$ of $A$.
\item $A$ is called $\overline{\sigma}$-$l$-miserable if $A$ is a countable union of closed $l$-miserable sets.
\end{enumerate}
\end{definition}

An immediate consequence of \cite[Theorem 1.5.]{BM} is the following result.

\begin{theorem}\label{banmas} \ ( T. Banakh and O. Maslyuchenko, 2020)

Let $X$ be a separable Banach space and $M \subset X$. Then the following conditions are equivalent.
\begin{enumerate}
\item  $M= D(f)$ for some linearly continuous function $f$ on $X$ which has the Baire property.
\item   $M$ is $\overline{\sigma}$-$l$-miserable.
\end{enumerate}
\end{theorem}
 Note that if $X= \R^n$, then each linearly continuous function on $X$ has the Baire property by
 \eqref{leb} and so Theorem \ref{banmas} gives a full characterization of  sets of discontinuities of
 linearly continuous functions on $\R^n$.

\section{Main result}\label{mare}

In this section we prove our main Theorem  \ref{hlavni} using the following basic lemma, whose rather easy proof is based on the Banakh-Maslyuchenko characterization (Theorem \ref{banmas}).

\begin{lemma}\label{DP}
Let $f \in C^1[0,1]$, let $\emptyset \neq P \subset (0,1)$ be a perfect nowhere dense set and let
$D \subset P $ be a countable dense subset of $P$.  Let, for each $d \in D$, two points  $u_d, v_d \in (0,1)$
 be given such that  $u_d < v_d < d$,
\begin{equation}\label{uvnitr}
 (d,f(d)) \in \interior  T_{f, [u_d,v_d] \cap P}
\end{equation}
and
\begin{equation}\label{Dep}
\text{the set \ $D_{\ep}:= \{ d \in D: d-u_d >\ep\}$\ is finite for each $\ep>0$.}
\end{equation}
Then there does not exist any linearly continuous function $g$ on $\R^2$ which is discontinuous
 at each point of the set  $\graph (f|_P)$.
\end{lemma}
\begin{proof}
Suppose to the contrary that such a function $g$ exists. By Theorem \ref{banmas} there
 exists a $\overline \sigma$-$l$-miserable set $A \subset \R^2$ such that $\graph (f|_P) \subset A$.
 Let $A_1, A_2,\dots$ be closed $l$-miserable subsets of $\R^2$ such that $A= \bigcup_{n=1}^{\infty} A_n$.
 Since the set  $\graph (f|_P)$ is closed in $\R^2$, by the Baire theorem there exists $k \in \N$
 such that the closed set $A_k \cap \graph (f|_P)$ is not nowhere dense in $\graph (f|_P)$ and so there exists
 an interval $(a,b) \subset [0,1]$ such that $P \cap (a,b) \neq \emptyset$ and 
 $\graph (f|_{P\cap (a,b)}) \subset A_k$. Since $A_k$ is $l$-miserable, we can choose a closed $l$-neighbourhood
 $L$ of $A_k$ such that $A_k \subset \overline H$, where $H:= \R^2 \setminus L$. 
Further choose $K>0$ such that $\|f'\| \leq K$.

 Now we will construct inductively a sequence
 of intervals $[a_n,b_n]$, $n=0,1,\dots$, such that for each $n\geq0$ the following two conditions hold.
\smallskip

(C1)\ \ $[a_n,b_n] \subset (a,b)$,\ \ $(a_n,b_n) \cap P \neq \emptyset$\ \ and\ \ $n (b_n-a_n) < 1$.

\smallskip

(C2)\ \ If $n\geq 1$, then $[a_n, b_n] \subset (a_{n-1},b _{n-1})$ and for each $x \in [a_n,b_n]$
 there exists a point $z_x^n \in T_{f,x} \cap H$ such that $\|z_x^n - (x,f(x))\| < 3(K+1)/n$. 
\smallskip

We can clearly choose  $[a_0,b_0]$ such that condition (C1) holds for $n=0$.

Further suppose that $n\geq 1$ and we have defined $[a_{n-1},b_{n-1}]$ such that
 condition (C1) holds for ``$n=n-1$''. 
Choose $p_n \in P \cap (a_{n-1},b_{n-1})$ and $\delta_n>0$ such that
 $[p_n- \delta_n, p_n+ \delta_n] \subset   (a_{n-1},b_{n-1})$. Since $P$ is perfect, the set
 $D \cap (p_n- \delta_n, p_n+ \delta_n)$ is infinite and so by \eqref{Dep} we can choose 
$d_n \in D \cap (p_n- \delta_n, p_n+ \delta_n)$ such that $[u_{d_n}, v_{d_n}] \subset (a_{n-1},b_{n-1})$ 
 and $d_n- u_{d_n} < 1/n$.

We know  that  $(d_n,f(d_n)) \in \graph f|_{P \cap (a,b)}  \subset A_k \subset \overline{H}$. Thus,
 since  $(d_n,f(d_n)) \in \interior  T_{f, [u_{d_n},v_{d_n}] \cap P}$ 
by \eqref{uvnitr} and $H$ is open, we can choose an open set $W \neq \emptyset$ such that
\begin{equation}\label{W}
  W \subset B((d_n,f(d_n)), 1/n) \cap H \cap  T_{f, [u_{d_n},v_{d_n}] \cap P}.
	\end{equation}
Consequently we can choose $x_n \in [u_{d_n},v_{d_n}] \cap P$ for which $T_{f,x_n} \cap W \neq \emptyset$.
Since $f \in C^1[0,1]$ and $W$ is open, it is easy to see that there exists an open neighbourhood
 $(a_n,b_n)$ of $x_n$ such that  $[a_n,b_n] \subset (a_{n-1},b_{n-1})$, $b_n-a_n < 1/n$ and
 $T_{f,x} \cap W \neq \emptyset$ for each $x \in [a_n,b_n]$. So we can choose for each 
 $x \in [a_n,b_n]$ a point $z_x^n = (\alpha_x^n, \beta_x^n)   \in T_{f,x_n} \cap W$.
Obviously  (C1) holds and  $[a_n, b_n] \subset (a_{n-1},b _{n-1})$ and thus it is sufficient to check
 that  $\|z_x^n - (x,f(x))\| < 3(K+1)/n$. To this end first observe that $|x- \alpha_x^n|<3/n$ since
 we easily see that $|d_n - \alpha_x^n|<1/n$, $|d_n- x_n| < 1/n$ and $|x_n-x|< 1/n$. Since the absolute value
 of the slope of the tangent $T_{f,x}$ is at most $K$, we obtain that 
$$ \|z_x^n - (x,f(x))\| \leq \sqrt{|x- \alpha_x^n|^2 + (K |x- \alpha_x^n|)^2} \leq (1+K) |x- \alpha_x^n|  < 3(K+1)/n.$$
 So we have finished our inductive costruction.

Now observe that the closedness   of  $P$ and condition (C1)  imply that $\bigcap_{n=0}^{\infty}
 [a_n,b_n] =  \{p\}$ for some $p\in P\cap (a,b)$. Applying (C2) to $p$ for ech $n\geq 1$ we obtain 
 points  $z_p^n \in T_{f,p} \cap H$,  $n\geq 1$, such that $z_p^n \to (p,f(p))\in A_k$. Consequently  $L=\R^2 \setminus H$ is not
 an $l$-neigbourhood of $A_k$, which is a contradiction.
\end{proof}
We will need also the following technical lemma.

\begin{lemma}\label{darb}
 Let $G$ and $\tilde G$ be functions from $C^1[0,1] $ and $g:= G'$, $\tilde g: =(\tilde G)'$.
 Let numbers $0<u<z< v < x<1$, $y\in \R$ and $\ep>0$, $\delta>0$, $\eta>0$ have the following properties:
\begin{equation}\label{cisla}
v + \delta + \eta <x ,\ \ v-u < \eta  \ \ \text{and}\ \ 6 \eta< \ep\, \delta,
\end{equation}
 \begin{equation}\label{Tzxy}
 (x,y) \in T_{G,z},
\end{equation}
\begin{equation}\label{ggt}
\|G - \tilde G\| \leq \eta,
\end{equation}
\begin{equation}\label{og}
\|g\| \leq 1\ \ \ \text{and}\ \ \ \osc(g, [u,v]) \leq \eta,
\end{equation}
\begin{equation}\label{zjzd}
\text{ there exist $u< s_1<s_2< v$ such that  $\tilde g(s_1) = g(s_1) - \ep$ and 
 $\tilde g(s_2) = g(s_2) + \ep$}.
\end{equation}
Then
\begin{equation}\label{obskou}
B((x,y), \eta) \subset T_{\tilde G, (u,v)}.
\end{equation}
\end{lemma}
\begin{proof}
Consider an arbitrary point $(\overline x, \overline y) \in B((x,y), \eta)$. Then $x- \eta < \overline  x < x+ \eta$
 and $y- \eta < \overline  y < y+ \eta$. 
Let $s_1$, $s_2$ be as in \eqref{zjzd}. Set
$$  h(s):= A_{\tilde G,s}(\overline x) = \tilde G(s) + \tilde g(s) \cdot (\overline x-s),\ \ \ s \in [s_1,s_2].$$
It is sufficient to prove that
\begin{equation}\label{hzjd}
h(s_1) \leq y - \eta\ \ \ \ \text{and}\ \ \ h(s_2) \geq y + \eta.
\end{equation}
Indeed, \eqref{hzjd} and the continuity of the function $h$ imply that there exists $\overline s \in (s_1,s_2)$
 such that  $h(\overline s)= \overline y$; consequently $(\overline x, \overline y) \in T_{\tilde G, (u, v)}$
 and \eqref{obskou} follows.

Recall that  $h(s_1) = \tilde G(s_1) + \tilde g(s_1) \cdot (\overline x-s_1)$,\ $h(s_2) = \tilde G(s_2) + \tilde g(s_2) \cdot (\overline x-s_2)$. By   \eqref{Tzxy},
\begin{multline*}
y= G(z) + g(z) (x-z) = G(z) + g(s_1)(x-z) + (g(z)-g(s_1))(x-z) \\
= G(z) + g(s_1)((\overline x -s_1) + (x - \overline x) + (s_1-z)) + (g(z)-g(s_1))(x-z).
\end{multline*}
Using these equalities, $\eqref{cisla}$, $\eqref{ggt}$, $\eqref{og}$, the choice of $s_1$, $s_2$
    and the inequality $|G(z)- G(s_1)| <  \eta$  which  follows from  \eqref{og} and $|z-s_1| < \eta$    by the mean value theorem, we obtain 
\begin{multline*}
y - h(s_1)= (G(z) - G(s_1)) + (G(s_1) - \tilde G (s_1)) + (g(s_1) - \tilde g(s_1)) (\overline x-s_1) +
  g(s_1) (x - \overline x)\\ + g(s_1) (s_1-z) + (g(z)-g(s_1))(x-z) 
	 \geq - \eta  -\eta + \ep \delta - \eta-  \eta  - \eta  \geq  \eta.
\end{multline*}
Quite analogously we obtain $y- h(s_2) \leq -\eta$ and 	$\eqref{hzjd}$ follows.	
\end{proof}

The proof of our main Theorem \ref{hlavni} is based on the construction of $f \in C^1[0,1]$,
 $P \subset [0,1]$ (and also  $D$  and $u_d$, $v_d$)  which satisfy assumptions of Lemma \ref{DP}.
 We will set  $f:= \lim_{n\to \infty} f_n$ and $P:= \bigcap_{n=1}^{\infty} P_n$, where $(f_n)$
 and $(P_n)$   are defined by a nontrivial inductive construction, in which each $P_n$ is a finite union of compact intervals. For sets $P_n$ we will use the following notation.

\begin{definition}\label{prop}
If $\emptyset \neq P \subset \R$ is a finite union of nondegenerate compact intervals, we 
\begin{enumerate}
\item 
denote by $\cal C(P)$ the set of all components of $P$,  and set
\item  
$R(P):=\{d \in \R:\ d\ \text{is a right endpoint of some}\ I \in \cal C(P)\}$,
\item 
$\nu(P):= \max\{\lambda(I):\ I \in \cal C(P)\}.$
\end{enumerate}
\end{definition}

\begin{theorem}\label{hlavni}
There exist $f \in C^1[0,1]$ and a closed set $M \subset \graph f$ which is nowhere dense in $\graph f$
 such that there does not exist any linearly continuous $g$ on $\R^2$ which is discontinuous at each point of
  $M$.
\end{theorem}
\begin{proof}
We will  define sequences  $(\eta_k)_{k=1}^{\infty}$,  $(P_k)_{k=1}^{\infty}$,  $(f_k)_{k=1}^{\infty}$
such that the following seven conditions hold for each $k\in \N$.
\begin{equation}\label{eta}
 \eta_1=1\ \  \text{and}\ \ 0< \eta_k < \frac{ \eta_{k-1}}{2}
,\ \ k\geq 2.
\end{equation}
\begin{equation}\label{pktvar}
\emptyset \neq  P_k \subset (0,1)\ \ \text{is a finite union of nondegenerate compact  intervals}.
\end{equation}
\begin{equation}\label{pkvl}
\nu(P_k) \leq \frac{1}{k},\ P_k \subset P_{k-1} \ \ \text{and}\ \ 
R(P_{k-1})\subset R(P_k)\ \ \text{if}\ \ k\geq 2.
\end{equation}
\begin{equation}\label{pkvl2}
 \text{Each}\ \  I \in \cal C(P_{k-1})\ \ \text{contains at least two elements of}\ \ \cal C(P_{k})\
\ \text{if}\ \ k\geq 2.
\end{equation} 
\begin{equation}\label{fktvar}
 f_k \in C^1[0,1]
\ \ \text{and}\ \ f'_k\ \ \text{has infinite variation on each interval}\ \ [\alpha,\beta] \subset [0,1].
\end{equation}
\begin{equation}\label{bis}
f_k(0) = 0\ \ \text{and}\ \ \|f_1'\| \leq \frac{1}{2}.
\end{equation}
\begin{equation}\label{fkvl}
\| f_{k}- f_{k-1}\| < \frac{\eta_k}{2}\ \ \text{and}\  \ \| f'_{k}- f'_{k-1}\| = 2^{-k} \ \ \text{if}\ \ k\geq 2.
\end{equation}
Moreover, for each $k\in \N$ and and each point $d$ from the set $ R^*_k$, where
\begin{equation}\label{defrh}
 R^*_1:= R(P_1)\ \ \text{and}\ \    R^*_k:= R(P_k) \setminus  R(P_{k-1})\ \ \text{for}\ \ k\geq 2,
\end{equation}
 we will define an interval $[u_d,v_d]$ such that the following two conditions hold:
\begin{equation}\label{uvpr}
  0<u_d<v_d<d,\  \ [u_d,d] \subset P_k      \ \ \ \text{and}\ \ \  3 \eta_{k+1} < d-v_d.
\end{equation} 
\begin{equation}\label{uvdr}
\text{If}\ k<l,\  \text{then}\ \  \ \ \ 
T_{f_l, (u_d,v_d) \cap \interior P_l}  \supset  B((d,f_k(d)), \eta_{k+1}).
\end{equation}
In the formulation of \eqref{uvpr} and \eqref{uvdr} we have used that  $R_1^*, R_2^*, \dots$ are 
 pairwise disjoint by \eqref{pkvl} and so $k$ is uniquely determined by $d$.
\medskip

In our inductive construction we will have defined,
after the $n$ th step  ($n\in \N$) of the construction, the numbers  $\eta_1,\dots,\eta_n$, the sets $P_1,\dots, P_n$, the functions
 $f_1,\dots, f_n$ and, if $n\geq 2$, for each  $1\leq k \leq n-1$ and $d \in R_k^*$ (see \eqref{defrh}), we will have defined an interval $[u_d,v_d]$,
 such that 
\begin{equation}\label{kdon}
 \text{seven conditions}\  \eqref{eta} - \eqref{fkvl}
\ \ \text{hold whenever}\ \ 1\leq k \leq n,
\end{equation} 
 \begin{equation}\label{kdonmj}
 \text{conditions}\   \eqref{uvpr}\ \text{and}\  \eqref{uvdr}
 \ \ \text{hold whenever}\ \ 1\leq k \leq n-1, \  d \in R^*_k\   \text{and}\  l\leq n
\end{equation} 
and
\begin{multline}\label{dz}
\text{for any point}\  d \in R^*_n \   
\text{there exists}\ 0< e < d\\
 \text{such that}\ [e,d] \subset P_n, \  e \in \interior P_n \   \text{and}\ (d,f_n(d)) \in T_{f_n, e}.
\end{multline}

\medskip

{\bf The first step.}
\smallskip

We set $\eta_1:= 1$. Choose (using e.g. \cite[Corollary 22, p. 143]{Br})  a  nowhere differentiable function $g \in C[0,1]$ with $\|g\| \leq 1/2$ and 
 set $f_1(x):= \int_0^{x} g,\ x \in [0,1]$. Then $f'_1 = g$ has infinite variation on each interval $[\alpha, \beta] \subset [0,1]$.
 Using Lemma \ref{tlgr} with $f:=f_1$, $\alpha:=0$ and $\beta:=1$, we can choose $0<e< w <1$ such that  $(w,f_1(w))\in T_{f_1,e}$
 and set $P_1:= [e/2, w]$. It is easy to check that   conditions   \eqref{kdon},  
   \eqref{kdonmj} and   \eqref{dz} hold for $n=1$.

\medskip

{\bf The inductive step.}
\smallskip

We suppose that $m\geq 2$ and the $(m-1)$ th step of the construction was accomplished.  In particular, we
 know that  conditions   \eqref{kdon},     \eqref{kdonmj} and   \eqref{dz} hold
 for $n=m-1$.

Our aim is to construct $\eta_m$, $P_{m}$, $f_{m}$ and an interval $[u_d,v_d]$ for each
 $d \in R_{m-1}^*$ such that  \eqref{kdon},     \eqref{kdonmj} and   \eqref{dz} hold
 for $n=m$.

First we choose, by the validity of \eqref{dz}
 for $n=m-1$, for each $d \in R^*_{m-1}$ a point $e=:z_d$  such that
\begin{equation}\label{ezd}
  0<z_d <d,\    [z_d,d] \subset P_{m-1}, \  z_d \in \interior P_{m-1} \  \ \text{and}\ \ (d,f_{m-1}(d)) \in T_{f_{m-1}, z_d}
\end{equation}
  and set 
	$$   
	Z^m_1:= \{z_d:\ d \in R^*_{m-1}\}.$$
	\smallskip
	 Further choose $\delta_m>0$ so small, that
	\begin{equation}\label{delta}
	\delta_m< \frac{d-z_d}{3} \ \ \text{for each}\ \ d \in R_{m-1}^*\ 
	 \text{and} \  \delta_m< \min\left(\frac{1}{2m}, \eta_{m-1}\right).
	\end{equation}

	Now we set $Z^m_2:=\emptyset$ if $m=2$ and, if $m\geq 3$, we define  $Z^m_2$ as follows.
	 In this case $R^*_k \neq \emptyset$ for each $1\leq k \leq m-2$ (see \eqref{kdon} and \eqref{pkvl2})
	 and for  such $k$ and $d \in R^*_k$ we have defined an interval $(u_d,v_d)$ such that, by the validity
	  \eqref{kdonmj} for $n=m-1$,
	\eqref{uvpr} holds and 
	\begin{equation}\label{prommj}
	T_{f_{m-1}, (u_d,v_d) \cap \interior P_{m-1}}  \supset  B((d,f_k(d)), \eta_{k+1}).
	\end{equation}
	 Choose $\eta_m>0$ so small that 
	\begin{equation}\label{eta2}
	 6 \eta_m < 2^{-m} \delta_m, \ \  \eta_m < \frac{\eta_{m-1}}{2} \ \ \text{and}\ \
	\end{equation}
	\begin{equation}\label{eta3}
	 6 \eta_m < 2^{-m} \frac{d-v_d}{3},\ \ \text{whenever}\ \ 1\leq k \leq m-2 \ \ \text{and}\ \ d \in R^*_k.
	\end{equation}
	Further,
	for every fixed $1\leq k \leq m-2$ and $d \in R^*_k$,  we choose a finite $\eta_{m}$-net  $Q_{d}$ of the ball
	  $B((d,f_k(d)), \eta_{k+1})$ and for each $q\in Q_{d}$ choose  by \eqref{prommj} a point
		$z_{q,d} \in (u_d,v_d) \cap \interior P_{m-1}$ such that $q \in T_{f_{m-1}, z_{q,d}}$.
		Now we define  $Z^m_2:= \{z_{q,d}:\   1\leq k \leq m-2,\ d \in R^*_k,\  q\in Q_{d}\}.$
		
		(Note that, as above, $k$ is uniquely determined by $d$; however  our construction alows 
		 cases when  $(q_1,d_1) \neq (q_2,d_2)$ and
		$z_{q_1,d_1} = z_{q_2,d_2}$.)
		
	Further choose $Z^m_3$  as an arbitrary finite set $Z^m_3\subset \interior P_{m-1} $ such that
	\begin{equation}\label{z3nepr}
 Z^m_3 \cap \interior I \neq \emptyset\ \ \text{for  each}\ \  I \in \cal C (P_{m-1})
\end{equation}
 and set 
	 $Z^m:= Z^m_1 \cup Z^m_2 \cup Z^m_3$.

	 Choose  $0<\delta_m^*< \delta_m$ so small that
	\begin{equation}\label{nadb}
	[z_{q,d}-\delta_m^*, z_{q,d}+\delta_m^*] \subset (u_d,v_d)\ \ \text{whenever}\ \ 1\leq k \leq m-2,\ d \in R^*_k\ \ \text{and}\ \  q\in Q_{d},
	\end{equation}
	\begin{equation}\label{list}
	\text{the intervals}\   \{ [z-\delta_m^*, z + \delta_m^*]:\ z \in Z^m\} \
	\text{ are pairwise disjoint subsets of}\   \interior P_{m-1},
	\end{equation}
	\begin{equation}\label{sjmale}
	\lambda \left( \bigcup_{z \in Z^m} [z-\delta_m^*, z + \delta_m^*]\right) < \frac{ \eta_m}{2},\ \ \text{and}
	\end{equation}
	\begin{equation}\label{osc}
	\osc(f'_{m-1}, [z-\delta_m^*, z+\delta_m^*]) \leq \eta_m\ \ \text{for each}\ \ z \in Z^m.
	\end{equation}
	Further choose a piecewise linear function $h_m
	\in C[0,1]$ such that
	\begin{equation}\label{hm}
	 \|h_m\| = 2^{-m},\ \ \supp h_m \subset \bigcup_{z \in Z^m} [z-\delta_m^*, z + \delta_m^*]\ \  \text{ and}
	\end{equation}
	\begin{multline}\label{zjzd2}
	\text{for every} \ z \in Z^m \ \text{ there exist}\  z- \delta_m^* <s_1^{z}< s_2^z < z    \\
		 \text{with}\ \ h_m(s_1^z)= -2^{-m},\  h_m(s_2^z)= 2^{-m}.
			\end{multline}
			Now we define $f_m$ by
			\begin{equation}\label{deffm}
			f_m(x) = f_{m-1}(x) + \int_0^x h_m,\ \ x \in [0,1].
			\end{equation}
			Then clearly  $f_m \in C^1[0,1]$ and, using \eqref{fktvar} for $k=m-1$, it is easy to
			 see that $f'_m´= f'_{m-1} + h_m$ has infinite variation on each interval $[\alpha,\beta] \subset [0,1]$.
			 
		So	by Lemma \ref{tlgr} we can find for each $z \in Z^m$ points $ z<  e_z < w_z < 
		 z + \delta_m^*$ such that
	 \begin{equation}\label{vw}
	(w_z,f_m(w_z)) \in T_{f_m,e_z}.
	\end{equation}
	For each $d \in R^*_{m-1}$, set
	\begin{equation}\label{defuv}
	u_d:= z_d - \delta_m^*,\ \ v_d:= w_{z_d}.
	\end{equation}
	
	To define $P_m$, assign to each  $d\in R(P_{m-1})$  a point  $c_d <d$ such that $c_d \in \interior P_{m-1}$,
	 $[c_d,d] \subset P_{m-1}$, $d-c_d < 1/m$ 
	 and $[c_d,d] \cap \bigcup_{z \in Z^m} [z-\delta_m^*, z + \delta_m^*] = \emptyset$,  and define
	\begin{equation}\label{defpm}
	P_m:= \bigcup_{z \in Z^m} [z-\delta_m^*, w_z] \cup \bigcup_{d\in R( P_{m-1})} [c_d,d].
	\end{equation}
	Thus we have constructed $\eta_m$, $f_m$, $P_m$, and an interval $[u_d,v_d]$ for each $d \in R^*_{m-1}$.
	
	\smallskip

	Our aim is now to prove that properties  \eqref{kdon}, \eqref{kdonmj} and \eqref{dz}
	 hold for $n=m$.
	
	First note that, by the above construction, 
		\begin{equation}\label{dvatroj}
		\eqref{defpm}\  \text{gives the decomposition of}\  P_m\ \text{into its components}.
		\end{equation}
		So, using \eqref{delta}, $\delta_m^* < \delta_m$ and $d-c_d < 1/m$ ($d \in R(P_{m-1})$), we obtain
		\begin{equation}\label{nipm}
		\nu(P_m) \leq 1/m.
		\end{equation}

	To prove that  \eqref{kdon} holds for $n=m$, it is sufficient (since we know that \eqref{kdon} holds for $n=m-1$)
	 to verify that conditions  \eqref{eta}-\eqref{fkvl} hold for $k=m$. These facts
	easily follow from the construction:
	
	Conditions
	\eqref{eta}   and  \eqref{pktvar} follow  from \eqref{eta2}  and \eqref{defpm}, respectively.
	Condition \eqref{pkvl} follows from \eqref{nipm},  \eqref{dvatroj} and   \eqref{list}.
	Condition \eqref{pkvl2} follows from  \eqref{dvatroj} and \eqref{z3nepr}.
	Condition \eqref{fktvar} is stated just after  \eqref{deffm}.
	Condition \eqref{bis} follows from \eqref{deffm} and the validity of \eqref{bis} for $k=m-1$.
	
	To prove 
	\eqref{fkvl}, observe that by \eqref{deffm} and \eqref{hm}  we have $\|f'_m - f'_{m-1}\| = \|h_m\| = 2^{-m}$ and, using
	 also \eqref{hm} and \eqref{sjmale}, we obtain $$|(f_m- f_{m-1})(x)| = \left| \int_{0}^x h_m\right| 
	 \leq \left|\int_{\supp h_m} |h_m| \right| < \frac{\eta_m}{2},\ \ \ x \in [0,1],$$ 
	 and so $\|f_m-f_{m-1}\| < \eta_m/2$.
	\smallskip
	
	Now we will show that \eqref{kdonmj} 
	 holds for $n=m$. Since we know that  \eqref{kdonmj} 
	 holds for $n=m-1$, it is sufficient to verify that
	\begin{equation}\label{uvpr2}
	\eqref{uvpr}  \ \text{holds if}\ k=m-1\ \text{and}\  d \in R_k^*
	\end{equation}
	and
	\begin{equation}\label{uvdr2}
	\eqref{uvdr} \ \ \text{holds if}\ \ 1\leq k \leq m-1,\ \ l=m\ \ \text{and}\ \ d\in R_k^*.
	\end{equation}
	
	To prove \eqref{uvpr2}, consider an arbitrary $d \in R_{m-1}^*$ and recall that $z_d \in Z^m_1 \subset Z^m$
	 and $u_d = z_d - \delta_m^*$, $v_d = w_{z_d}$ (see \eqref{defuv}. By \eqref{ezd} we
	 have $[z_d, d] \subset P_{m-1}$ 
	and so  \eqref{list}
	 and $z_d < w_{z_d} < z_d + \delta_m^*$ imply $0<u_d < v_d < d$ and $[u_d, d] \subset P_{m-1}$.
	 To prove $3 \eta_m < d-v_d$, observe that \eqref{delta} gives $d-z_d > 3 \delta_m$ and so, using
	 also $\delta_m^*< \delta_m$ and \eqref{eta2}, we obtain
	$$ d-v_d > d-z_d -\delta_m^* > 3 \delta_m - \delta_m >  3 \eta_m.$$
	So \eqref{uvpr2}
	 is proved.
	\smallskip

	To prove \eqref{uvdr2}, we will distinguish  cases {\bf a)} \ $k= m-1$ and {\bf b)} \ $1 \leq k < m-1$.
	\smallskip
	
	 {\bf a)}\ \ \ Consider an arbitrary
	$d \in R_{m-1}^*$.  Then  $z_d \in Z^m_1$
	 and $u_d = z_d - \delta_m^*$, $v_d = w_{z_d}$. By \eqref{ezd} we have
	\begin{equation}\label{rdel}
	(d,f_{m-1}(d)) \in T_{f_{m-1}, z_d}.
	\end{equation}
	Now we will show that the assumptions of Lemma \ref{darb} are satisfied for 
	$$ G= f_{m-1},\ \tilde G= f_m,\ u= u_d,\ z= z_d, \ v= v_d,\ x= d,\ y= f_{m-1}(d),\ \ep= 2^{-m},\ 
	 \delta= \delta_m,\ \eta= \eta_m.$$
	First we show that inequalities \eqref{cisla} hold:
	
	Using \eqref{eta2} and \eqref{delta} we obtain
	$$ v_d + \delta_m + \eta_m < (z_d+ \delta_m^*) + \delta_m + \frac{\delta_m}{6} < z_d + 3 \delta_m < d.$$
	Further we obtain  $v_d-u_d < 2 \delta_m^* < \eta_m$ by \eqref{sjmale} and
	 $ 6 \eta_m < 2^{-m} \delta_m$ by \eqref{eta2}.

	Condition \eqref{Tzxy} coincides with  \eqref{rdel} and \eqref{ggt} holds since we know that
	\eqref{fkvl} is valid for $k=m$. Condition \eqref{og} follows from \eqref{osc} and
	 the validity of \eqref{bis} and \eqref{fkvl} for each $ k \leq m-1$. Finally, condition
	 \eqref{zjzd} follows from \eqref{zjzd2} since $\tilde g - g = h_m$.  
	
	Thus conclusion \eqref{obskou} of Lemma \ref{darb} holds, i.e. 
	$B((d, f_{m-1}(d)), \eta_m) \subset T_{f_m, (u_d,v_d)}$. Since  $(u_d,v_d) \subset \interior P_m$   
	 by   \eqref{defuv} and \eqref{defpm}, we obtain that
	\eqref{uvdr} holds for
	 $k= m-1$, $l=m$ and our $d$. 
	\smallskip
	
	 {\bf b)}\ \ \ Consider  arbitrary $1 \leq k < m-1$ and
	$d \in R_k^*$. 
	 Then we have defined a finite  $\eta_{m}$-net  $Q_{d}$ of the ball
	  $B((d,f_k(d)), \eta_{k+1})$ and for each $q\in Q_{d}$ we have defined a point
		$z_{q,d} \in (u_d,v_d) \cap \interior P_{m-1}$ such that $q \in T_{f_{m-1}, z_{q,d}}$. 
		 
		 Now we will show that, for an arbitrary $q=:(x_q,y_q)\in Q_d$,  
		 the assumptions of Lemma \ref{darb} are satisfied for 
	\begin{multline*}
	G= f_{m-1},\ \tilde G= f_m,\ u=z_{q,d}- \delta_m^* ,\ z= z_{q,d}, \ v= w_{z_{q,d}}, \\ x= x_q,\ y= y_q,\ \ep= 2^{-m},\ 
	 \delta= \frac{d-v_d}{3}, \eta= \eta_m.
	\end{multline*}
	First we show that inequalities \eqref{cisla} hold:
	
	Using \eqref{nadb}, \eqref{eta3} and \eqref{uvpr} we obtain
	\begin{multline*}
	v+ \delta+ \eta = w_{z_{q,d}} + \frac{d-v_d}{3}  + \eta_m < v_d + \frac{d-v_d}{3} + \frac{d-v_d}{3} \\
	 = d - \frac{d-v_d}{3} < d - \eta_{k+1}  <  x_q =x.
	\end{multline*}
	Further we obtain the inequalities $ v-u = w_{z_{q,d}}- (z_{q,d}- \delta_m^* )  < 2 \delta_m^* < \eta_m = \eta$ by 
	\eqref{sjmale} and
	 $6 \eta =  6 \eta_m < 2^{-m}\frac{d-v_d}{3}= \ep \delta $ by \eqref{eta3}.

	Condition \eqref{Tzxy} holds since  $q \in T_{f_{m-1}, z_{q,d}}$  and \eqref{ggt} holds since we know that
	\eqref{fkvl} is valid for $k=m$. Condition \eqref{og} follows from \eqref{osc} and
	 the validity of \eqref{bis} and \eqref{fkvl} for each $ k \leq m-1$. Finally, condition
	 \eqref{zjzd} follows from \eqref{zjzd2} since $\tilde g - g = h_m$.

		 Thus assertion \eqref{obskou} of Lemma \ref{darb} holds, i.e.  
		$B(q,  \eta_m) \subset T_{f_m, (z_{q,d}- \delta_m^*,  w_{z_{q,d}})}$.
		 Note that
		 $ z_{q,d} \in Z_2^m \subset Z^m$ and so  $(z_{q,d}- \delta_m^*,  w_{z_{q,d}}) \subset (u_d,v_d) \cap   \interior P_m$ by \eqref{nadb} and \eqref{defpm}.

		Since    $Q_{d}$ is  a $\eta_{m}$-net of 
	  $B((d,f_k(d)), \eta_{k+1})$  
		 we obtain that \eqref{uvdr} holds for $l=m$ and our $k$ and $d$.
		\smallskip
		
		So we have proved \eqref{uvdr2}.
		It remains to prove that \eqref{dz} holds for $n=m$. So consider an arbitrary $d\in R_m^*$.
		 By \eqref{dvatroj} we obtain that there exists $z_d\in Z^m$ such that $d=w_{z_d}$
		 and   \eqref{vw}  shows that \eqref{dz} holds for $n=m$ (since the choice $e:= e_{z_d}$ works for our $d$).

		\smallskip

\smallskip

 So we have finished our inductive construction.   It is easy to see that we
 have defined the sequences  $(\eta_k)_{k=1}^{\infty}$,  $(P_k)_{k=1}^{\infty}$,  $(f_k)_{k=1}^{\infty}$
 and  intervals $[u_d,v_d]$ ($d\in R^*_k$, $k=1,\dots$) such that all nine properties
 \ \eqref{eta} - \eqref{uvdr}\ 
 hold for each $k \in \N$. 

Using these properties only, we will show that $f:= \lim_{n\to \infty} f_n$, $P:= \bigcap_{n=1}^{\infty} P_n$
 and $D:= \bigcup_{k=1}^{\infty} R_k^* = \bigcup_{k=1}^{\infty}  R(P_k)$ satisfy the assumptions of Lemma
 \ref{DP}. 

By \eqref{fkvl} we obtain that $\|f_m'-f_l'\| \leq 2^{-l}$ whenever $1 \leq l < m$ and therefore 
 the sequence $(f_k')$ uniformly converges to a function $\varphi \in C[0,1]$.
 Since $f_k(x)= \int_0^x f_k',\ x \in [0,1],$ by  \eqref{bis}, we obtain  $ f_k \rightrightarrows f$, where
 $f(x):= \int_0^x  \vf,\ x \in [0,1].$ 
 Clearly  $f \in C^1[0,1]$ and $ f'_k \rightrightarrows f'=\vf$.

By   \eqref{pktvar} and   \eqref{pkvl}, $P= \bigcap_{n=1}^{\infty} P_n$ is a nonempty closed set and  \eqref{pkvl}
 with  \eqref{pkvl2} easily imply that $P$ is perfect and nowhere dense. 

By \eqref{pkvl} we easily obtain that the countable set $D \subset P$ is dense in $P$.

Now consider an arbitrary $d\in D$. Then there exists $k \in \N$ such that $d \in R_k^*$ and so we have defined
 $u_d, v_d \in (0,1)$ such that $u_d < v_d <d$ and \eqref{uvpr} and \eqref{uvdr} hold.
 Now consider an arbitrary point  $(x,y) \in B((d,f_k(d)),\eta_{k+1})$. By  \eqref{uvdr} we can choose, for each
 $l>k$, a point $p_l \in [u_d,v_d] \cap P_l$ such that  $(x,y) \in T_{f_l,p_l}$. Choose a convergent subsequence
 $(p_{l_i})_{i=1}^{\infty}$ of $(p_l)_{l=k+1}^{\infty}$ with  $p_{l_i} \to p$. Then $p \in P \cap [u_d,v_d]$
 and Lemma \ref{stkon} gives $(x,y) \in T_{f,p}$. Thus we have proved that
\begin{equation}\label{hvbt}
  B((d,f_k(d)),\eta_{k+1}) \subset T_{f, P \cap [u_d,v_d]}.
	\end{equation}
	Using \eqref{fkvl} and \eqref{eta}, we obtain
	\begin{multline*}
	|f(d) - f_k(d)| \leq \|f_{k+1} - f_k\| + \|f_{k+2} - f_{k+1}\| + \dots\\
	< \eta_{k+1}/2 + \eta_{k+2}/2 + \eta_{k+3}/2+\dots \leq (1/2) (\eta_{k+1} + \eta_{k+1}/2  + \eta_{k+1}/4+ \dots)  = \eta_{k+1}.
	\end{multline*}
	So \eqref{hvbt} gives 
	$$ (d,f(d)) \in  B((d,f_k(d)),\eta_{k+1}) \subset \interior T_{f, P \cap [u_d,v_d]}$$
	 and \eqref{uvnitr} is proved.
	
	Finally, condition \eqref{Dep} holds since  each $R_k^*$ is finite and $d-u_d < 1/k$ for $d \in R_k^*$
	 by \eqref{pkvl} and  \eqref{uvpr}.

	So the assertion of our theorem holds for $M:= \graph f|_P$ by  Lemma \ref{DP}.
\end{proof}

\section{Slobodnik's necessary condition is not sufficient}\label{SNCINS}

As an easy consequence of Theorem \ref{hlavni}, we will prove the following result which shows that Slobodnik's necessary condition
 (for sets from $\cal D_l^n$) is not suficient.

\begin{proposition}\label{nepost}
There exists  a set $M\subset \R^2$ such that $M = \bigcup_{k=1}^{\infty} B_k$, where $B_k$
 have properties (i), (ii) and (iii) from Theorem \ref{Slob} for $n=2$, but $M \notin \cal D_l^2$.
\end{proposition}
\begin{proof}
Let $M$ and $f$ be as in Theorem \ref{hlavni}. So we have  $M \notin \cal D_l^2$ and we have $M = \{(x,f(x)):\ x\in P\}$ where $P\subset [0,1]$ is nowhere dense. 
 Set  $B_k:= M,\ k\in \N$; we will show that then properties (i), (ii), (iii) hold. 

The property (i) is almost obvious, since  we can extend $f$ to a Lipschitz function
 $f^*$ on $\R$ and thus $M$ is a compact subset of $\graph f^*$ which is a Lipschitz hypersurface in $\R^2$.

To prove (ii), consider a hyperplane $H$ in $\R^2$. Then $H$ is a line and we can suppose that it contains
 the origin. Then the projection onto $H$ is a linear mapping and so it is clearly sufficient to prove that,
 for each  $a,b \in \R$, 
 $Z:= \{ax+ bf(x):\ x \in P\}$     is nowhere dense in $\R$.  Since  $Z= g(P)$, where $g(x):= ax+f(x),\ x\in [0,1],$ is $C^1$-smooth 
 and $P$ is nowhere dense, we obtain (see \cite[Lemma 4.1]{CG1}) that $Z= g(P)$ is nowhere dense.

To prove (iii), consider an arbitrary  point $c=(c_1,c_2) \in \R^2 \setminus M$. To prove that the set
\begin{multline*}  E:= \left\{\frac{(x,y)-c}{\|(x,y)-c\|}:\  (x,y)\in M\right\} \\= \left\{\left(\frac{x-c_1}{\sqrt{(x-c_1)^2 + (f(x)-c_2)^2}},
 \frac{f(x)-c_2}{\sqrt{(x-c_1)^2 + (f(x)-c_2)^2}}\right),\ x\in P\right\}
\end{multline*}
is nowhere dense in $S_{\R^2}$, it is clearly sufficient to show that (its ``projection'')
$$ F:= \left\{\frac{x-c_1}{\sqrt{(x-c_1)^2 + (f(x)-c_2)^2}},\ x\in P\right\}$$
is nowhere dense in $\R$. 
Set  
$$h(x):= \frac{x-c_1}{\sqrt{(x-c_1)^2 + (f(x)-c_2)^2}}\ \ \ \text{  if}\ \ \   (x,f(x)) \neq c.$$
If $c_1 \in P$ or $c_1 \notin [0,1]$, then  $h$ is clearly $C^1$-smooth on $[0,1]$ and therefore
 $F= h(P)$ is nowhere dense by  \cite[Lemma 4.1]{CG1}. 

If  $c_1 \in [0,1] \setminus P$, then $h$ can be undefined for $x=c_1$. However, then there exist
 points   $0<t_1 < t_2 < 1$ such that  $F= h(P \cap [0,t_1]) \cup   h(P \cap [t_2,1])$ and
 $h$ is $C^1$-smooth on  $[0,t_1]$ (resp. $[t_2,1]$) whenever  $P \cap [0,t_1] \neq \emptyset$
 (resp.  $P \cap [t_2,1] \neq \emptyset$). So \cite[Lemma 4.1]{CG1} gives again that $F$ is nowhere dense.
\end{proof}
\begin{remark}\label{iiih}
In the above proposition, we can state that sets $B_k$ satisfy the following condition (more general than  (iii)):
\smallskip

 (iii)* \ \ \ {\it For each $c \in \R^2$, the set $\{\frac{x-c}{\|x-c\|}:\ x \in B_k\setminus \{c\}\}$ is nowhere dense in
 the unite sphere $S_{\R^2}$.}
\smallskip

The proof is only a slight refinement of the proof of (iii).

\end{remark}

\section{An analogon of Slobodnik's necessary condition in separabůe Banach spaces}\label{SNC}
In this section we will prove the following analogon of Slobodnik's  Theorem \ref{Slob}.
\begin{proposition}\label{anslo} 
Let $X$ be a separable Banach space ($\dim X \geq 2$) and let $f: X\to \R$ be a linearly continuous function having the Baire property. Then  the set $D(f)$
	of all discontinuity points of $f$  can be written as  $D(f)= \bigcup_{k=1}^{\infty} B_k$
	 where each $B_k$ has the following properties.
	\begin{enumerate}
\item  $B_k$ is a bounded closed subset of a  Lipschitz hypersurface $L_k \subset X$ which is nowhere dense in
 $L_k$.
\item Any linear projection of $B_k$ onto any hyperplane 
	   $0\in H \subset X$ is a first category subset  of $H$.
\item  For each $c \in  X \setminus B_k$, the set $\{\frac{x-c}{\|x-c\|}:\ x \in B_k\}$ is a first category set in  $S_X$.
\end{enumerate}
\end{proposition}

This result improves \cite[Corollary 4.2]{Za} which only asserts that $D(f)$ can be 
	covered by countably many Lipschitz hypersurfaces (and implies that $D(f)$  is a null subset of $X$
	 in any usual sense). 
	
	We infer Proposition \ref{anslo} easily from \cite[Corollary 4.2]{Za}, 
	the Banakh-Maslyuchenko characterization (Theorem \ref{banmas}) and simple Lemma \ref{proj} below.
	
	Although our Proposition \ref{anslo} is not a direct generalization of Slobodnik's Theorem \ref{Slob},
	  {\it  it easily implies Slobodnik's  result}.   Indeed, if $X= \R^n$, then each linearly continuous $f$ has the Baire property by 
	 \eqref{leb} and  all $B_k$ and their projections (in (ii) and (iii)) are compact. So  it is sufficient to use the fact that  each closed  first category subset of a complete metric space is nowhere dense.

\begin{lemma}\label{proj}
Let $X$ be a separable 
Banach space with $\dim X \geq 2$, $v \in S_X$ and let $Y$ be a topological complement of $V:= \lin\{v\}$.  
 Let $A \subset X$ be an $l$-miserable set. Then 
\begin{enumerate}
\item
 the projection  $\pi_Y(A)$ of $A$ onto $Y$
 in the direction of $V$ is of the first category in $Y$, and
\item
 for each $c \in X \setminus A$, the set $\{\frac{x-c}{\|x-c\|}:\ x \in A\}$ is of the first category in $S_X$.
\end{enumerate}
\end{lemma}
\begin{proof}
Choose a closed $l$-neighbourhood $L$ of $A$ such that $A \subset \overline{X \setminus L}$. The proofs of (i) and (ii) are quite analogous.
\smallskip

(i)\ \ For each $n\in \N$, set  $A_n:= \{x\in A:\ \{x+tv:\ t \in [-1/n, 1/n]\} \subset L\}$. Since
$L$ is an $l$-neighbourhood  of $A$, we have $A= \bigcup_{n \in \N} A_n$. 
 Now choose, for each $n \in \N$, a covering of $V$ by open sets  $V_{n,1}, V_{n,2},...$ such that
 each $V_{n,k} \subset V$ is  an open subset of $V$ with  $\diam(V_{n,k}) < 1/n$. 
 Denote by $\pi_V$ the projection of $X$ onto $V$ in the direction of $Y$ and set
 $A_{n,k}:= A_n \cap (\pi_V)^{-1}(V_{n,k}),\ n,k \in \N$. 
 Using the definitions of  sets $A_n$ and $V_{n,k}$,
it is easy to see that  
\begin{equation}\label{jevl}
(\pi_Y)^{-1}(\pi_Y(x)) \cap (\pi_V)^{-1}(V_{n,k})\subset L \ \ \text{whenever}\ \  n,k\in \N,\ x \in A_{n,k}.
\end{equation}
Since   $A = \bigcup_{n,k \in \N} A_{n,k}$, it is sufficient to prove that each set
 $\pi_Y(A_{n,k})$ is nowhere dense in the space $Y$. So suppose, to the contrary, that
 there exist  $n$, $k$ and a nonempty set $G\subset Y$ which is open in $Y$ and $\pi_Y(A_{n,k})\cap G$
 is dense in $G$. Using  \eqref{jevl} and the closedness of $L$, we obtain
 $ (\pi_Y)^{-1}(G) \cap (\pi_V)^{-1}(V_{n,k}) \subset L$. Since  $ (\pi_Y)^{-1}(G) \cap (\pi_V)^{-1}(V_{n,k})$
 is open
 and contains a point $ a\in A_{n,k} \subset A$, we obtain  $a \notin \overline{X \setminus L}$ which contradicts
 to  $A \subset \overline{X \setminus L}$.
\smallskip

(ii)\ \ Fix a point $c\in X\setminus A$ and define the ``projection''  $\pi_S: X\setminus \{c\} \to S_X$
 by  $\pi_S(x):= \frac{x-c}{\|x-c\|}$, $x \in  X\setminus \{c\}$.
 For each $n\in \N$, denote  $A_n:= \{x\in A:\ \{x+t \pi_S(x) :\ t \in [-1/n, 1/n]\} \subset L \}$. Since
$L$ is an $l$-neighbourhood  of $A$, we have $A= \bigcup_{n \in \N} A_n$. 
 Now choose, for each $n \in \N$, a covering of $(0,\infty)$ by open subsets  $H_{n,1}, H_{n,2},...$  with  $\diam(H_{n,k}) < 1/n$.
Set
 $A_{n,k}:= A_n \cap \{x\in X: \|x-c\|\in H_{n,k}\},\ n,k \in \N$. 
 Using the definitions of  sets $A_n$ and $H_{n,k}$,
it is easy to see that  
\begin{equation}\label{jevl2}
(\pi_S)^{-1}(\pi_S(x)) \cap \{x\in X: \|x-c\|\in H_{n,k}\}\subset L \ \ \text{whenever}\ \  n,k\in \N,\ x \in A_{n,k}.
\end{equation}
Since   $A = \bigcup_{n,k \in \N} A_{n,k}$, it is sufficient to prove that each set
 $\pi_S(A_{n,k})$ is nowhere dense in the sphere $S$. So suppose, to the contrary, that
 there exist  $n$, $k$ and a nonempty set $G\subset S$ which is open in $S$ and $\pi_S(A_{n,k})\cap G$
 is dense in $G$. Using  \eqref{jevl2} and the closedness of $L$, we obtain
 $ (\pi_S)^{-1}(G) \cap \{x\in X: \|x-c\|\in H_{n,k}\} \subset L$. Since $ (\pi_S)^{-1}(G) \cap \{x\in X: \|x-c\|\in H_{n,k}\}$
 is open
 and contains a point $ a\in A_{n,k} \subset A$, we obtain  $a \notin \overline{X \setminus L}$ which contradicts
 to  $A \subset \overline{X \setminus L}$.
\end{proof}

{\bf Proof of Proposition \ref{anslo}.}

By Theorem \ref{banmas} there exist closed $l$-miserable sets $F_m \subset X$, $m \in \N$, such that
 $D(f)= \bigcup_{m\in \N} F_m$.  By  \cite[Corollary 4.2]{Za} there exist Lipschitz hypersurfaces
 $M_n$, $n \in \N$, such that $D(f) \subset \bigcup_{n\in \N} M_n$.  Set
$$  Z_{m,n,p}: = F_m \cap M_n \cap \{x\in X:\ \|x\| \leq p\},\ \ \ 
m,n,p \in \N, $$
 and let $(B_k)_{k=1}^{\infty}$ be a sequence of all elements of the set  $\{Z_{m,n,p}:\ m,n,p \in \N\}$.
 Now consider an arbitrary $k \in \N$. Obviously,  $B_k$ is closed and bounded. It is also $l$-miserable
 since it is a subset of an $l$-miserable set $F_m$ and so we obtain by Lemma \ref{proj} that 
conditions (ii) and (iii) hold.
Further $B_k$ is
contained in
 some Lipschitz hypersurface $M_n$. Let  $V$ and $Y$ be linear spaces corresponding to
 $M_n$ as in Definition \ref{liphyp} and denote by $\pi_Y$ the projection onto $Y$ in the direction of $V$.
Since $\pi_Y|_{M_n}: M_n \to Y$ is clearly a homeomorphism, we obtain that $\pi_Y(B_k)$ is a closed
 subset of the complete space $Y$ and so it is nowhere dense in $Y$, because it is a first category set in $Y$
 by condition (ii). Consequently  $B_k$ is nowhere dense in $M_n$ and thus condition (i) holds (with $L_k:= M_n$).

\end{document}